\documentclass[11pt]{amsart}
\usepackage{amssymb,amsfonts,amsthm,amsmath,mathrsfs,xspace,hyperref}
\usepackage{graphicx}
\usepackage[francais,english]{babel}
\usepackage{inputenc}
\usepackage[T1]{fontenc}
\usepackage[centering]{geometry}
\usepackage{csquotes}
\usepackage{color}
\usepackage{enumerate}
\usepackage{enumitem}
  \newcommand{\R}{\ensuremath{\mathbb{R}}}%
  \newcommand{\Z}{\ensuremath{\mathbb{Z}}}%
	\newcommand{\Q}{\ensuremath{\mathbb{Q}}}%
  \newcommand{\N}{\ensuremath{\mathbb{N}}}%
	\newcommand{\F}{\ensuremath{\mathcal{F}}}%

				\newcommand{\X}{\ensuremath{\mathcal{X}}}%

        \renewcommand{\H}{\ensuremath{\mathcal{H}}}%
                \newcommand{\K}{\ensuremath{\mathcal{K}}}%

    \newcommand{\alt}{\ensuremath{\operatorname{Alt}}}%

  \newcommand{\sub}{\ensuremath{\operatorname{Sub}}}%
        \newcommand{\psl}{\ensuremath{\operatorname{PSL}}}%
        \renewcommand{\sl}{\ensuremath{\operatorname{SL}}}%

	\newcommand{\urs}{\ensuremath{\operatorname{URS}}}
		\newcommand{\irs}{\ensuremath{\operatorname{IRS}}}
				
	\newcommand{\prob}{\ensuremath{\mathcal{P}}}
	\newcommand{\cl}{\ensuremath{\mathcal{F}}}

\theoremstyle{definition}
  \newtheorem{defin}{Definition}[section]

  \newtheorem{question}[defin]{Question}

\theoremstyle{plain}
  \newtheorem{thm}[defin]{Theorem}
  \newtheorem{main thm}{Theorem}
  \newtheorem{prop}[defin]{Proposition}
    \newtheorem{prop-def}[defin]{Proposition-Definition}

  \newtheorem{cor}[defin]{Corollary}
   \newtheorem{lem}[defin]{Lemma}

\theoremstyle{remark}
  \newtheorem{remark}[defin]{Remark}
    \newtheorem{remarks}[defin]{Remarks}

  \newtheorem{example}[defin]{Example}
  
\begin{document}

  \date{June 26, 2018}	
\title{Locally compact groups whose ergodic or minimal actions are all free}

\author{Adrien Le Boudec}
\thanks{The authors were partially supported by ANR-14-CE25-0004 GAMME, and ALB was partially supported by ANR-12-BS01-0003-01-GDSous/GSG}
\thanks{This work was completed when ALB was F.R.S.-FNRS Postdoctoral Researcher at UCLouvain.}
\address{CNRS, UMPA - ENS Lyon, 46 all\'ee d'Italie, 69364 Lyon, France}
\email{adrien.le-boudec@ens-lyon.fr}

\author{Nicol\'as Matte Bon}
%\thanks{NMB was partially supported by Projet ANR-14-CE25-0004 GAMME}
\address{Departement of Mathematics, ETH Z\"urich, R\"amistrasse 101, 8092 Z\"urich, Switzerland}
\email{nicolas.matte@math.ethz.ch}

\begin{abstract}
We construct locally compact groups with no non-trivial Invariant Random Subgroups and no non-trivial Uniformly Recurrent Subgroups. 
\end{abstract}

\maketitle

\section{Introduction}

Let  $G$ be a locally compact, second countable group. We denote by $\sub(G)$ the space of closed subgroups of $G$, endowed with the Chabauty topology \cite{Chabauty-topo}. The space $\sub(G)$ is compact, and $G$ acts continuously by conjugation on $\sub(G)$. Fixed points for this action are precisely the closed normal subgroups of $G$.

An \emph{Invariant Random Subgroup} (IRS) of $G$ is a probability measure on $\sub(G)$ which is invariant under the action of $G$ \cite{A-G-V}. For instance any closed subgroup $H$ of finite covolume in $G$ (e.g.\ a lattice) gives rise to an IRS by considering the push-forward of an invariant probability measure on $G/H$ by the map $G/H \rightarrow \sub(G)$, $g H \mapsto gHg^{-1}$. The space $\irs(G)$ is closed for the {weak-$\ast$} topology in the space $\mathcal{P}(\sub(G))$ of probability measures on $\sub(G)$, and hence is a compact space. Viewing lattices as elements of this compact space led to recent developments, see e.g.\ \cite{7s-partI,Gel-KM}. More generally if $G$ acts on a probability space $(X, \mu)$ by measure preserving transformations, then the push-forward of $\mu$ through the stabilizer map $X \to \sub(G)$ is an IRS of $G$. In fact, every IRS arises in this way \cite{A-G-V,7s-partI}. %Every group $G$ admits two  obvious IRS's, namely the Dirac masses $\delta_1$ and $\delta_G$. 

The notion of IRS admits a natural topological counterpart, namely closed invariant subspaces of $\sub(G)$. We will denote by $\F(\sub(G))$ the compact space of closed subsets of $\sub(G)$, and by $\F(\sub(G))^G$ the subspace of closed invariant subsets. Among these, an important role is played by minimal ones. A minimal closed invariant subset of $\sub(G)$ is called a \emph{Uniformly Recurrent Subgroup} (URS) \cite{Glas-Wei}.  For instance if $H$ is a closed cocompact subgroup of $G$, then the conjugacy class of $H$ is closed in $\sub(G)$, and is therefore a URS. More generally any minimal action of $G$ on a compact space gives rise to a URS, called the stabilizer URS of the action \cite{Glas-Wei}. The converse of this statement is also true, namely that every URS arises as the stabilizer URS of some minimal action, see \cite{MB-Ts-realizing} (for finitely generated groups, this  was also established in \cite{Elek-urs}). %Every group $G$ admits the singletons $\{1\}$ and $G$ as URS's, and we will refer to these as \emph{trivial}.

In this article, we are interested in the following problems:

\begin{question} \label{q-i-no-IRS}
Are there locally compact groups such that every IRS is a convex combination of $\delta_1$ and $\delta_G$ ?
\end{question}

\begin{question} \label{q-i-no-URS}
Are there locally compact groups with no URS other than $\{ \! \{1\} \! \}$ and $\{ G \}$ ?
\end{question}

%Equivalently, a group $G$ satisfies the conclusion of Question \ref{q-i-no-IRS} if and only if  every ergodic probability measure preserving action of $G$ is essentially free. It satisfies the conclusion of Question \ref{q-i-no-URS}  if and only if every non-trivial minimal action of $G$ on a compact space is \emph{topologically free}, i.e. there is a dense $G_\delta$ set of points with trivial stabilizer.

We will say that $G$ has no non-trivial IRS's if every IRS is a convex combination of $\delta_1$ and $\delta_G$, and that $G$ has no non-trivial URS's if $\{ \! \{1\} \! \}$ and $\left\{G\right\}$ are the only URS's. Thanks to the aforementioned characterisation of IRS's and URS's in terms of stabilizers  \cite{7s-partI, MB-Ts-realizing}, $G$ has no non-trivial IRS's if and only if every non-trivial ergodic probability measure preserving action of $G$ is \emph{essentially free}, i.e. there is a full measure subsets consisting of points with trivial stabilizer; and $G$ has no non-trivial URS's if and only if every non-trivial minimal action on a compact space is \emph{topologically free}, i.e.\ there is a dense $G_\delta$ set of points with trivial stabilizer. 

In connection with Question \ref{q-i-no-URS}, it is interesting to compare locally compact groups with more general Polish topological groups. Recall that there exist Polish groups with no non-trivial minimal action on compact spaces (these are called \emph{extremely amenable}). From this point of view, locally compact groups as in Question \ref{q-i-no-URS} behave like extremely amenable groups in restriction to non-topologically free actions. Recall however that non-trivial locally compact groups are never extremely amenable, as they always admit a free minimal action on a compact space \cite{Ellis-free, Veech}. 

\bigskip

Groups having no non-trivial IRS's or no non-trivial URS's are known to exist among \textit{discrete groups}, and our focus in this paper will be about non-discrete groups, for which Questions \ref{q-i-no-IRS} and \ref{q-i-no-URS} are open. For IRS's this question appears in \cite{7s-partII}, and is discussed in more details in \cite{Gel-notes}. Since any closed normal subgroup is a fixed point for the action of $G$ on $\sub(G)$, any potential candidate for having no non-trivial IRS's or no non-trivial URS's must be topologically simple. A group with no non-trivial IRS's should also fail to admit any lattice. The first examples of compactly generated simple groups not containing any lattice were the Neretin groups \cite{BCGM}. Other examples, arising as groups acting on trees with almost prescribed local action, have been exhibited in \cite{LB-ae}. Recently, coloured versions of the Neretin groups have been introduced and studied in \cite{Lederle}. It is not known whether the families of groups from \cite{BCGM,LB-ae,Lederle} contain instances with no non-trivial IRS's. Note however that they do admit non-trivial URS's (coming from their action on the boundary of the tree associated to them).

The goal of this article is to provide a simple construction of non-discrete groups answering Questions \ref{q-i-no-IRS} and \ref{q-i-no-URS} simultaneously. The examples that we give are not compactly generated, and both questions remain open for non-discrete compactly generated groups (see the discussion below).

\subsection*{Groups of piecewise affine homeomorphisms}

%A homeomorphism $g$ of $\Q_p$ is \emph{piecewise affine} if there exists a partition of $\Q_p$ into compact open subsets $(\X_i)_{i \geq 0}$ such that for all $i \geq 0$ the action of $g$ on $\X_i$ is given by a map of the form $x \mapsto a_i x + b_i$, with $a_i \in \Q_p^{*}$ and $b_i \in \Q_p$.

We denote by $\mathrm{PL}(\Q_p)$ the group of \emph{piecewise affine homeomorphisms} of the local field $\Q_p$. These are homeomorphisms $g$ such that there exists partition of $\Q_p$ into compact open subsets $(\mathcal{U}_i)_{i \geq 0}$ such that for all $i \geq 0$ the restriction of $g$ to $\mathcal{U}_i$ coincides with the restriction of a map of the form $x \mapsto a x + b$, with $a \in \Q_p^{*}$ and $b \in \Q_p$ (note that the sets $\mathcal{U}_i$ are not required to be $g$-invariant).  

Recall that two locally compact groups are  \emph{locally isomorphic} if they have isomorphic open subgroups.

\begin{thm} \label{thm-intro}
For every countable product of finite groups $U$, there exists a totally disconnected locally compact group $G \leq \mathrm{PL}(\Q_p)$ that is locally isomorphic to $U$ and such that:
\begin{enumerate}[label=(\roman*)]
 \item \label{item-no-irs} the only IRS's of $G$ are convex combinations of $\delta_1$ and $\delta_G$;
	\item \label{item-no-urs} $G$ has no URS other than $\{ \! \{1\} \! \}$ and $\left\{G\right\}$.
\end{enumerate}
\end{thm}

%This is illustrated by the examples of $\mathrm{PSL}(d,\Q)$ inside $\mathrm{PSL}(d,\Q_p)$, or by the Higman-Thompson's group $V$ inside the Neretin group $\mathcal{N}$ \red{\cite{Cap-dMe} (?)}. While the countable groups $\psl(d, \Q)$ and $V$ do not admit non-trivial IRS's \cite{Pet-Thom, Dud-Med}, this is not true for the group $\mathrm{PSL}(d,\Q_p)$, and is still unknown for the group $\mathcal{N}$. The main point of this paper is to provide a construction for which results on IRS's and URS's of discrete groups can be applied in order to study IRS's and URS's of non-discrete groups.

We refer to \S \ref{subsec:LC-PL} for an explicit description of groups as in Theorem \ref{thm-intro}. In connection with the previous analogy with extremely amenable Polish groups, we point out that these examples are not amenable.

All the locally compact groups $G$ that we construct share a common dense countable subgroup $\Lambda$, such that the groups $G$ are Schlichting completions of $\Lambda$ \cite{schlichting}. For the proof of the absence of IRS's in $G$, we invoke results of Dudko--Medynets \cite{Dud-Med} implying that the \textit{discrete group} $\Lambda$ has no non-trivial IRS's. Using simple approximation arguments, we then deduce the absence of IRS's for the non-discrete group $G$. The proof of the absence of URS's follows the same strategy, by appealing to results from \cite{LBMB-subdyn} concerning discrete groups. However we point out that in general results on IRS's or URS's of a countable group do \textit{not} pass to its completions. This is for instance illustrated by the example of $\mathrm{PSL}(3,\Q)$, which has no non-trivial IRS's (see Section \ref{s:more}), but whose completions $\mathrm{PSL}(3,\Q_p)$ do have non-trivial IRS's; for instance lattices. The main point of this paper is to provide a construction for which results on IRS's and URS's of discrete groups can be profitably used in order to study IRS's and URS's of non-discrete groups.

%We also point out that the groups $G$ obtained in the proof of Theorem \ref{thm-intro} are non-amenable. This fact can be compared with the aforementioned analogy with extremely amenable Polish groups.

This construction admits some variations. In Section \ref{s:more} we discuss two of them, arising as completions of the finitary alternating group $\alt_f(\Z)$, and of the group of infinite matrices $\sl(\infty , \Q)$. These provide other examples answering (separately) Questions \ref{q-i-no-URS} and  \ref{q-i-no-IRS}.

%\subsection{About compact generation}

\subsection*{About compact generation} \label{s-non-cg}

%It is inherent to our argument that the examples of groups from Theorem \ref{thm-intro} that we construct are not compactly generated. While candidates to answer Question \ref{q-i-no-IRS} for compactly generated groups are available \cite{BCGM,LB-ae}, the situation is rather different for URS's. A compactly generated non-discrete group with no non-trivial URS's should be totally disconnected and topologically simple, i.e.\ it should belong to the class $\mathcal{S}$ studied by Caprace--Reid--Willis \cite{CRW-partI,CRW-partII}, and all known examples of groups in $\mathcal{S}$  do admit non-trivial URS's. We refer to Appendix A from \cite{CRW-partII} for a list of currently known sources of examples of groups in $\mathcal{S}$, and to \cite{Capr-proc} for a survey about recent developments.

It is inherent to our argument that the examples of groups from Theorem \ref{thm-intro} that we construct are not compactly generated. While candidates to answer Question \ref{q-i-no-IRS} for compactly generated groups are available \cite{BCGM,LB-ae}, the situation is rather different for URS's. A compactly generated non-discrete group with no non-trivial URS's should be totally disconnected and topologically simple, i.e.\ should belong to the class $\mathcal{S}$ studied by Caprace--Reid--Willis \cite{CRW-partI,CRW-partII}. We do not know any groups in $\mathcal{S}$ that are candidates for answering Question \ref{q-i-no-URS}. We refer to Appendix A from \cite{CRW-partII} for a list of currently known sources of examples of groups in $\mathcal{S}$, and to \cite{Capr-proc} for a survey about recent developments.

An intermediate problem is whether there exist groups without non-trivial URS's which are non-elementary in the sense of Wesolek \cite{Wes-elem}.

\subsection*{Organization}

The article is organized as follows. In the next section we introduce notation and recall basic facts about the Chabauty topology. In Section \ref{s:as-limits} we consider an approximation process of a locally compact group, and explain how the study of IRS's and URS's of the ambient group can be reduced to the study of certain IRS's and URS's of the approximating subgroups. We apply these results in Section \ref{s:PL}, and construct the groups from Theorem \ref{thm-intro}. Finally in Section \ref{s:more} we give additional examples of non-discrete groups having no non-tivial IRS's or URS's.

\subsection*{Acknowledgments}

We are grateful to Tsachik Gelander for suggesting the problem of the existence of locally compact groups without IRS's. We thank Uri Bader, Pierre-Emmanuel Caprace, Tsachik Gelander, Jean Raimbault for interesting discussions. We also thank Pierre-Emmanuel Caprace and Phillip Wesolek for drawing our attention to the construction from \cite{Willis-co} and \cite{A-G-W}, a variant of which is considered in \S \ref{subsec:permutations}. Finally we thank an anonymous referee for remarks and corrections.

We thank the Isaac Newton Institute for Mathematical Sciences, Cambridge, for hospitality during the programme \textit{Non-positive curvature group actions and cohomology}, where part of this work was undertaken.

\section{Preliminaries}

%\red{For some  statements second countability will not be essential (especially when no measure theory is involved), but we will not attempt to specify it.}

% Throughout the paper we use the term \emph{locally compact group} as a shortcut for ``locally compact, Hausdorff, second countable topological group''.

In this article all locally compact groups $G$ are assumed to be second countable. We denote by $\sub(G)$ the set of closed subgroups of $G$, endowed with the Chabauty topology. Recall that a pre-basis of open sets for this topology is given by sets of the form
\[\mathcal{U}_V=\{H\in \sub(G)\colon H\cap V\neq \varnothing\} \quad \mathcal{O}_C=\{H\in \sub(G) \colon H\cap C=\varnothing\},\]
where $V$ is a relatively compact open subset of $G$, and $C$ a compact subset of $G$. The space $\sub(G)$ is compact and metrisable (as we assume $G$ second countable). The convergence in $\sub(G)$ can be characterized as follows: a sequence $(H_n)$ converges to $H$ in $\sub(G)$ if and only if every $h\in H$ is the limit of a sequence $(h_n)$ where $h_n\in H_n$ for every $n$, and conversely every cluster point of such a sequence belongs to $H$.

%The group $G$ acts on $\sub(G)$ by conjugation. 

The following facts are well-known (see e.g.~\cite{Schochetman} for proofs).

\begin{lem} \label{lem-chab-prelim}
Let $G$ be a locally compact group.
\begin{enumerate}[label=(\roman*)]
\item \label{i:inclusion} Let $H$ be a closed subgroup of $G$. Then the inclusion $\sub(H)\to \sub(G)$ is continuous.
\item Let $N$ be a closed normal subgroup of $G$, and $\pi: G \rightarrow G/N$ the canonical projection. Then the map $\sub(G/N) \rightarrow \sub(G)$, $H \mapsto \pi^{-1}(H)$, is continuous. If moreover $N$ is compact, $\sub(G) \rightarrow \sub(G/N)$, $H \mapsto \pi(H)$, is also continuous.
\item \label{i:trunc} If $O\le G$ is open, the intersection map $\sub(G)\to \sub(G), \, H\mapsto H\cap O$, is continuous.
\item If $(H_n)$ is a decreasing sequence of closed subgroups, then $(H_n)$ converges to $\cap H_n$.
\item If $(H_n)$ is increasing, then $(H_n)$ converges to $\overline{\cup H_n}$.
\end{enumerate}
\end{lem}

Let $X$ be a compact metrisable space. A map $\varphi\colon X\to \sub(G)$ is \emph{upper semicontinuous} if for every sequence $(x_n)\subset X$ converging to a limit $x\in X$, every cluster point $K$ of $(\varphi(x_n))$ in $\sub(G)$ verifies $K\le \varphi(x)$. We recall the following classical fact.

\begin{lem}\label{l-semicontinuous-measurable}
An upper semicontinuous map $\varphi \colon X\to \sub(G)$ is measurable for the Borel $\sigma$-algebras on $X$ and on $\sub(G)$. 
\end{lem} 

\begin{proof}
It is enough to show that $\varphi^{-1}(\mathcal{U}_V)$ and $\varphi^{-1}(\mathcal{O}_C)$ are measurable whenever $V\subset G$ is open and $C\subset G$ is compact. First, we claim that $\varphi^{-1}(\mathcal{O}_C)$ is  actually open, or equivalently that $\varphi^{-1}(\mathcal{O}_C^c)=\{x\in X\colon \varphi(x) \cap C\neq \varnothing\}$ is closed. Let $(x_n)\subset \varphi^{-1}(\mathcal{O}_C^c)$ be a sequence converging to some  $x\in X$. Choose a sequence $g_n\in\varphi(x_n)\cap C$. Upon passing to a subsequence, $(g_n)$ converges to $g\in C$. Then every cluster point $K$ of $\varphi(x_n)$ contains $g$, and by upper semicontinuity we get that $g\in \varphi(x)$, proving that $x\in \varphi^{-1}(\mathcal{O}_C^c)$.

Now let $V\subset G$ be an open set. Since $G$ is locally compact and second countable we may write $V$ as the countable union of compact subsets $C_n$, and we have $\mathcal{U}_V=\bigcup \mathcal{O}_{C_n}^c$, showing that $\varphi^{-1}(\mathcal{U}_V)$ is also Borel. \qedhere
\end{proof}

If $X$ is a compact space we denote by $\prob(X)$ the convex space of all probability measures on $X$, endowed with the weak-* topology, and by $\cl(X)$ the set of closed subsets of $X$, endowed with the Hausdorff topology. Recall that $\prob(X)$ and $\cl(X)$ are metrizable whenever $X$ is metrizable. 

%This is in particular the case for $\mathcal{P}(\sub(G))$ and $\cl(\sub(G))$ for a second countable locally compact group $G$.

Statement \ref{i:inclusion} of Lemma \ref{lem-chab-prelim} implies that whenever $H\le G$ is a closed subgroup, the inclusion  $\sub(H)\subset \sub(G)$ induces closed inclusions $\prob(\sub(H)) \subset \prob(\sub(G))$ and $\cl(\sub(H)) \subset \cl(\sub(G))$. In the sequel we will always use these identifications without further mention.

The conjugation action of $G$ on $\sub(G)$ induces continuous actions on $\prob(\sub(G))$ and $\cl(\sub(G))$. We denote by $\irs(G) \subset \prob(\sub(G))$ the space of $G$-invariant probability measures. It is a closed subspace of $\prob(\sub(G))$, and hence a compact space for the induced topology. Similarly we let $\cl(\sub(G))^G\subset \cl(\sub(G))$ be the set of $G$-invariant closed subsets of $\sub(G)$. Again $\cl(\sub(G))^G$ is compact for the induced topology. We denote by $\urs(G) \subset \cl(\sub(G))^G$ the uniformly recurrent subgroups, i.e.\ minimal $G$-invariant closed subsets of $\sub(G)$. Note that in general $\urs(G)$ is \emph{not} closed  in $\cl(\sub(G))$ (this is for instance the case for the lamplighter group, as can be seen from the analysis of URS's of wreath products in  \cite[\S 3]{Glas-Wei}).

%or instance this can be deduced from the analysis of URS's of wreath products made in

We will say for short that a group $G$ has \emph{no non-trivial IRS's} if the only IRS's of $G$ are convex combinations of $\delta_1$ and $\delta_G$, and that $G$ has \emph{no non-trivial URS's} if its only URS's are $\{1\}$ and $G$.

\section{Approximations in the Chabauty space} \label{s:as-limits}

\subsection{Bi-approximations}

The approximation process considered in Sections \ref{s:PL} and \ref{s:more} consists of a locally compact group $G$ which can be written as an ascending union of open subgroups $G_n$, and which contains a decreasing sequence $(U_n)$ such that $U_n$ is a compact open normal subgroup of $G_n$. In this section we work in a slightly more general setting (in which the subgroups $U_n$ need neither be open nor normal in $G_n$), and explain how the study of IRS's and URS's of the group $G$ can be reduced to the study of IRS's and URS's of $G_n$ \enquote{relatively to} the coset space $G_n/U_n$ (in a sense made precise below).

%The groups $G$ that we will consider to prove Theorem \ref{thm-intro}, the other examples considered in Section \ref{s:more}  share the property that they can be written as the ascending union of a sequence $(G_n)$ of open subgroups, and admit a decreasing sequence $(U_n)$ of compact open subgroups such that $U_n\unlhd G_n$ for all $n$. 

%In this section we preliminary explain how, in this situation, the study of IRS's and URS's of the group $G$ can be reduced to the study of IRS's and URS's of the sequence of groups $G_n$ that come from the quotient $G_n/U_n$. 

%We work in a slightly more general setting in which the groups $U_n$ need not be open, nor normal in $G_n$, clarified in the definition below. In this case, it turns out that it is enough to study IRS's and URS's of $G_n$ ``relatively to'' the coset space $G_n/U_n$, in a sense that will be made precise later.  We will indicate when assuming the normality of $U_n$ simplifies the statements. 

\begin{defin} \label{def-bi-app}
Let $G$ be a locally compact group. A \emph{bi-approximation}  for $G$ is a sequence $(G_n, U_n)$ of pairs of closed subgroups $U_n \leq G_n \leq G$ such that:
\begin{enumerate}[label=(\roman*)]
\item $(G_n)$ is an increasing sequence of open subgroups such that $G=\cup G_n$;
\item $(U_n)$ converges to $\{1\}$ in $\sub(G)$, and $\cup U_n$ is relatively compact in $G$. 
\end{enumerate}
\end{defin}

\begin{remark} \label{rmq-case-Un-monot}
When the sequence of compact subgroups $(U_n)$ is decreasing and $\cap U_n = 1$; then the second condition of Definition \ref{def-bi-app} is satisfied (see Lemma \ref{lem-chab-prelim}).
\end{remark}

%This particular case actually covers all the examples that we study in details in Sections \ref{s:PL} and \ref{s:more}.

Yet another way to formulate the condition on the sequence $(U_n)$ is the following:

%Here is a reformulation of the second condition of Definition \ref{def-bi-app}.

\begin{lem} \label{lem-reformul-U}
Let $(U_n)$ be sequence of compact subgroups that is relatively compact in $G$. The following are equivalent:
\begin{enumerate}[label=(\roman*)]
\item $(U_n)$ converges to $\{1\}$ in $\sub(G)$;
\item every sequence $(u_n)$ with $u_n\in U_n$ converges to $1$ in $G$.
\end{enumerate}
\end{lem}

\begin{proof}
If $(U_n)$ converges to $\{1\}$ and $u_n\in U_n$ for all $n$, then any cluster point of $(u_n)$ should belong to the limit of $(U_n)$, and hence is trivial. Therefore the relatively compact sequence $(u_n)$ admits $1$ as unique cluster point, and hence converges to $1$. The converse implication is clear.
\end{proof}

Before moving further, let us indicate that we do not suppose $(G_n)$ to be strictly increasing in Definition \ref{def-bi-app}, so that all results of this section hold true in the extreme case $G_n=G$ for all $n$. Similarly the $U_n$ are not necessarily pairwise distinct; and $U_n=\{1\}$ for all $n$ is also allowed.

\subsection{Truncation and saturation maps}

We will need the following terminology.

\begin{defin}
Let $G$ be a group and $U\le G$ be a subgroup. For a subgroup $H\le G$, the $U$-\emph{saturation} of $H$ in $G$, denoted $[H]^G_{U}$, is defined as the subgroup of $G$ consisting  of all elements whose action on $G/U$  preserves every $H$-orbit. \end{defin}

It is clear from this definition that $[H]_U^G$ is a subgroup of $G$. The following lemma provides a more concrete description of $[H]_U^G$ and records some of its basic properties.
\begin{lem} \label{l-saturation-basic}
Let $G$ be a group and $U, H\le G$ be subgroups.
\begin{enumerate}[label=(\roman*)]
\item \label{i-saturation-formula} The $U$-saturation of $H$ is given by the formula:
 \begin{equation} \label{e:intersection}   [H]_U^G=\bigcap_{g\in G} H gUg^{-1}. \end{equation} 
\item \label{i-double-saturation} We have {$[[H]^G_{U}]^G_{U}=[H]_U^G$}. 
\item \label{i-saturation-normal} If $U$ is normal in $G$, then $[H]_U^G=HU$.
\end{enumerate}
\end{lem}
\begin{proof}
Let $k\in [H]_U^G$ and let $g \in G$. Since $k$ preserves the $H$-orbit of the coset $gU$, we deduce that there exists $h\in H$ such that $kgU=hgU$,  from which it  follows that $k\in HgUg^{-1}$. Hence $ [H]_U^G\subset \bigcap_{g\in G} H gUg^{-1}$, since $k\in[H]_U^G$ and $g\in G$ are arbitrary. Conversely, for every $g\in G$ the subset $HgUg^{-1}$ preserves the $H$-orbit of $gU$. Therefore every element of $\bigcap_{g\in G} H gUg^{-1}$ preserves every $H$-orbit in $G/U$, i.e. $\bigcap_{g\in G} H gUg^{-1}\subset [H]_U^G$. This proves \ref{i-saturation-formula}.

 Observe that $H\le [H]_U^G$, and thus the  $[H]_U^G$-orbits in $G/U$ coincide with the $H$-orbits. This readily implies \ref{i-double-saturation}. Finally, \ref{i-saturation-normal} follows from \ref{i-saturation-formula}. \qedhere

\end{proof}

We will say that $H$ is $U$-\emph{saturated} if $[H]^G_{U}=H$ (note that this depends not only on $U$ but also on the ambient group $G$). When $U$ is clear from the context, we will simply say that $H$ is saturated. 

%Note that if the group $G$ has a topology, in general $[H]_U^G$ needs not be closed even if $H$ an $U$ are closed. 

If $U$ is normal in $G$, then by part \ref{i-saturation-normal} of Lemma \ref{l-saturation-basic}, the set of saturated subgroups consists of preimages of subgroups of the quotient $G/U$. This motivates the terminology. 

%\begin{defin}
%If $G$ is a locally compact group and $U\le G$ is closed, the set of closed subgroups that are saturated relatively to $G/U$ will be denoted $\sub(G, G/U)$. \red{ $\sub(G)_{G/U}$ ?}
%\end{defin}

%(but not necessarily normal)

\begin{remark}
Let $\mathcal{P}$ be a partition of the coset space $G/U$ and denote by $G_\mathcal{P}$ the subgroup of $G$ that preserves individually each subset of the partition $\mathcal{P}$. Then $G_\mathcal{P}$ is saturated, and every saturated subgroup is of this form.
\end{remark}

%Assume that $U$ is open in $G$, so that the space $G/U$ is discrete. 
%In particular, every saturated subgroup is closed.

\begin{lem} \label{lem-saturated-closed}
Let $G$ be a locally compact group, and $U\le G$ a compact subgroup. Then $[H]_U^G$ is closed in $G$ for every closed subgroup $H \leq G$.
\end{lem}

\begin{proof}
For every $g\in G$, the set  $HgUg^{-1}$ is the product of a closed subgoup and a compact subgroup, and hence is a closed subset of $G$. It follows that $[H]_U^G=\bigcap_g HgUg^{-1}$ remains closed. \qedhere
\end{proof}

\begin{defin}
When $U$ is a compact subgroup of $G$, the map $\sub(G)\to \sub(G),\, H\mapsto [H]_U^G$, will be called the \emph{saturation map} relative to $G/U$.
\end{defin}

The next lemma says that the saturation map $H\mapsto [H]_U^G$ is well behaved with respect to the Chabauty topology. 
\begin{lem} \label{l:saturation-map}
Let $G$ be a locally compact group, and $U\le G$ a compact subgroup. The following hold:
\begin{enumerate}[label=(\roman*)]

%\item \label{i:sat-closed} for every closed subgroup $H\le G$, the subgroup $[H]_U^G$ is closed in $G$;

\item \label{i:sat-measurable} the saturation map $H \mapsto [H]_U^G$ is equivariant.
\item \label{i:sat-up-sc} For every $H_n$ converging to $H$ and every cluster point $K$ of $[H_n]_U^G$, we have $H \leq K \leq [H]_U^G$. In particular $H \mapsto [H]_U^G$ is upper semicontinuous, and measurable.
\item \label{i:sat-continuous} If $U$ is moreover assumed to be normal, then $H\mapsto [H]_U^G$ is continuous.
\end{enumerate}
\end{lem} 

%in particular $\sub(G, {G/U})$ is closed in $\sub(G)$

%\red{la derniere phrase induisait un peu en erreur je trouve, car le fait que $\sub(G, {G/U})$ est fermé est étranger à la continuité de l'application.}

\begin{proof}
We denote by $\varphi$ the map from the statement.

\ref{i:sat-measurable}. The fact that $\varphi$ is equivariant is apparent from the formula \eqref{e:intersection}. 

\ref{i:sat-up-sc}. Let $(n_k)$ such that $\varphi(H_{n_k})$ converges to $K$. Since $H_{n_k} \leq \varphi(H_{n_k})$ for all $k$ and $H_{n_k}$ converges to $H$, we have $H \leq K$. Moreover for every $g\in G$, the subset $H_{n_k} gUg^{-1}$ converges to $H gUg^{-1}$ in the Hausdorff topology on all closed subsets of $G$. Since $\varphi(H_{n_k})\subset H_{n_k} gUg^{-1}$, it follows that $K$ is contained in $HgUg^{-1}$ for all $g\in G$, and therefore $K$ is contained in $\varphi(H)$. This shows upper semicontinuity, and measurability follows (Lemma \ref{l-semicontinuous-measurable}).

\ref{i:sat-continuous} follows from Lemma \ref{lem-chab-prelim}. \qedhere
\end{proof}

%To see that $\varphi$ is upper semicontinuous, let $(H_\nu)$ be a net converging to $H$.

The image of the saturation map $[\, \cdot \,]_U^G$ is exactly the set of saturated subgroups. We warn the reader that it is not a closed subset of $\sub(G)$ in general. This is illustrated by the following example, which shows simultaneously that the map $[\, \cdot \,]_U^G$ is not continuous in general. 

%However when $U$ is normal, it coincides with the set of subgroups containing $U$, and hence is closed in $\sub(G)$.

\begin{example}
Suppose that we have $G = H \rtimes U$, where $U$ is a non-trivial compact group such that for every $u \neq 1$, the set of all commutators $[h,u]$, when $h$ varies in $H$, is the entire $H$. If $H$ is not isolated in $\sub(H)$, then the set of saturated subgroups is not closed in $\sub(G)$, and $[\, \cdot \,]_U^G$ is not continuous. 

Indeed, let $(H_n)$ be a sequence of proper subgroups of $H$ converging to $H$, and let $\gamma \in [H_n]_U^G$. Since $\gamma \in H_n U$, one may write $\gamma = h_n u_0$. Additionally for every $h \in H$, there exists $k_n \in H_n$ and $u \in U$ such that $k_n \gamma = h u h^{-1}$, i.e.\ $k_n h_n u_0 = [h,u] u$. Necessarily we have $u=u_0$, and $k_n h_n = [h,u_0]$. Therefore $[H,u_0]$ lies in a proper subgroup of $H$, which forces $u_0$ to be trivial by our assumption. So we deduce $\gamma = h_n$ and finally $[H_n]_U^G = H_n$. Therefore $H_n$ is saturated for all $n$, while $[H]_U^G = G$ shows that the limit $H$ is not saturated. A concrete example is given by $G = \mathbb{Z}[1/2] \rtimes \left\{ \pm 1\right\}$ and $H_n = \frac{1}{2^n} \mathbb{Z}$ for $n \geq 1$.
\end{example}

%$H = \mathbb{R}^2$, $U = \mathrm{SO}(2,\mathbb{R})$ and $H_n = \mathbb{R} \times \frac{1}{n} \mathbb{Z}$ for $n \geq 1$.

%***** \red{j'ai déplacé ce paragraphe, voir quoi en faire}
%
%In particular, if $G$ is the ascending union of $G_n$, we have a sequence of maps $\sub(G)\to \sub(G), H\mapsto H\cap G_n$. We call these maps the \emph{truncation maps} associated to the sequence $(G_n)$. It is not difficult to see that $H\cap G_n$ converges to $H$ in the Chabauty topology. We now want to define a better sequence of maps in the case when $(G_n, U_n)$ is a bi-approximation of $G$ and prove an analogue of this fact.
%
%
%*****

%
%\subsection{s and trunc-saturation maps}
%We introduce the following terminology. 
%\begin{defin}
%Let $G$ be a locally compact group. We say that $G$ is the \emph{ascending-shrinking} limit of  $(G_n, U_n)$ if $G_n$ is an increasing sequence of open subgroups of $G$ such that $G=\cup G_n$ and $U_n\le G_n$  is a sequence of compact subgroups such that 
%\begin{enumerate}[label=(\roman*)]

%\item the sequence $(U_n)$ tends to $\{1\}$ in $\sub(G)$;
%\item the union $\bigcup U_n$ is relatively compact in $G$. 
%\end{enumerate}
%\end{defin}
%\begin{remark}\label{r:u-to-1}
%The two conditions on the sequence $(U_n)$ are equivalent to the following: every sequence $(u_n)$ with $u_n\in U_n$ converges to $1$ in $G$.
%\end{remark}

\begin{defin}
Let $(G_n, U_n)$ be a bi-approximation of $G$. The maps
\[
\lambda_n\colon \sub(G)\to \sub(G), \, H\mapsto [H\cap G_n]^{G_n}_{U_n}
\]
will be called the \emph{trunc-saturation} maps associated  $(G_n, U_n)$.
\end{defin}

%\red{
%The following lemma sums up the relevant basic properties of the trunc-saturation maps, that follow by combining Lemmata \ref{lem-chab-prelim} \ref{i:trunc} and \ref{l:saturation-map}.}

For future reference we summarize the following basic properties:

\begin{prop} \label{prop:trunc-saturation}
For every $n$, the  map $\lambda_n$ is Borel-measurable, upper semicontinuous, $G_n$-equivariant (but not necessarily $G$-equivariant), and takes values in the set of closed $U_n$-saturated subgroups of $G_n$. Moreover $\lambda_n$ is continuous if $U_n\unlhd G_n$.
\end{prop}

\begin{proof}
Note that the map $H \mapsto H \cap G_n$ must be continuous since $G_n$ is open in $G$ (Lemma \ref{lem-chab-prelim} \ref{i:trunc}), so all the properties follow from Lemmas \ref{lem-saturated-closed} and \ref{l:saturation-map}.
\end{proof}

The following fact will be used repeatedly. 

%In the last sentence, we view $\sub(G)$ as a uniform space with the unique uniformity compatible with its compact topology (note that since $G$ is second countable, $\sub(G)$ is metrisable, and therefore the uniformity is induced by any metric).

\begin{lem}\label{l:ts-maps}
Let $(G_n, U_n)$ be a bi-approximation of $G$, and let $\lambda_n\colon \sub(G)\to \sub(G)$ be the associated  trunc-saturation maps.  Then the sequence of maps $(\lambda_n)$ converges uniformly, as $n\to \infty$, to the identity map $\sub(G)\to \sub(G)$.
\end{lem}

 %Since we are assuming that $G$ is second countable, the space $\sub(G)$ is metrisable and therefore we may fix a metric $d$ an   that the convergence is uniform with respect to  $d$ on $\sub(G)$. If this is not the case, then we can find $\varepsilon \geq 0$ and a sequence $(n_i)$ of integers tending to $\infty$ and a sequence $H_{n_i}\in \sub(G)$ such that $d(H_{n_i}, \lambda_{n_i}(H_{n_i})\geq \varepsilon$. 

\begin{proof}
Since $\sub(G)$ is metrisable, the statement can equivalently be formulated as follows: for every sequence $H_n$ converging to $H$ in $\sub(G)$, the sequence $\lambda_n(H_n)$ also converges to $H$. Let us show that this holds.

Let $h\in H$. Since $H_{n}$ tends to $H$, we can choose a sequence $h_{n}\in H_{n}$ that converges to $h$. If $n$ is large enough, we can assume $h\in G_{n}$ hence also $h_{n}\in G_{n}$ for $n$ sufficiently large. It follows that eventally $h_{n}\in H_{n}\cap G_{n}$, which is contained in $\lambda_{n}(H_{n})$ and therefore $h$ belongs to any cluster point of $(\lambda_{n}(H_{n}))$. Conversely let $h_{n}\in \lambda_{n}(H_{n})$ be a sequence converging to some $h\in G$. Since $\lambda_n(H)\subset HU_n$ for all $n$ and all $H\in \sub(G)$, we can write $h_{n}=h'_{n}u_{n}, h'_{n}\in H_{n}, u_{n}\in U_{n}$, and we have $u_{n}\to 1$ (Lemma \ref{lem-reformul-U}). It follows that also $h'_{n}\in H_{n}$ converges to $h$. Therefore $h\in H$. \qedhere

% To see that the sequence $(\lambda_n)$ converges uniformly to the identity, let $\mathcal{U}\subset X\times X$ be a neighbourhood of the diagonal, and let us show that $(H, \lambda_n(H))\in \mathcal{U}$ holds for all large enough $n$ and all $H\in \sub(G)$.  If this is not the case, then we can find a sequence $(n_i)$ of integers tending to $\infty$ and a sequence $H_{n_i}\in \sub(G)$ such that $(H_{n_i}, \lambda_{n_i}(H_{n_i}))\notin \mathcal{U}$. Upon extracting we assume that $(H_{n_i})$ tends to a limit $H\in \sub(G)$. By the first statement (applied to the ascending-shrinking chain $(G_{n_i}, U_{n_i})$), we have $\lambda_{n_i}(H_{n_i})\to H$, contradicting that $(H_{n_i}, \lambda_{n_i}(H_{n_i}))\notin \mathcal{U}$.  
\end{proof}

\subsection{Bi-approximations and probability measures on $\sub(G)$}

Assume $(G_n, U_n)$ is a bi-approximation of $G$, and denote by $\lambda_n$ the associated trunc-saturation maps. Measurability and equivariance of $\lambda_n$ (Proposition \ref{prop:trunc-saturation}) imply that $\lambda_n$ induces a map $\lambda_n^\ast\colon \prob(\sub(G))\to \prob(\sub(G))$ by pushforward of measures, which remains $G_n$-equivariant (but not necessarily $G$-equivariant).

%For simplicity we will still call the maps $\lambda_n^\ast$ trunc-saturation maps.

\begin{defin}
Let $G$ be a locally compact group and $U\le G$ be a compact subgroup. We say that  $\mu\in \irs(G)$ is $U$-\emph{saturated}  if $\mu$-almost every $H\in \sub(G)$ is $U$-saturated.
\end{defin}

\begin{remark}
When $U$ is normal in $G$, an IRS of $G$ is $U$-saturated if and only if it is the pushforward of an IRS of $G/U$ under the natural map $\sub(G/U) \to \sub(G)$, and the set of $U$-saturated IRS's is naturally identified with $\irs(G/U)$.
\end{remark}

%H\mapsto \pi^{-1}(H)$

%\begin{remark}
%Each map $\lambda_n^*\colon \prob(\sub(G))\to \prob(\sub(G))$ is $G_n$-equivariant, but not  $G$-equivariant in general. \end{remark}

\begin{prop} \label{p:no-IRS-limit}
Let $(G_n, U_n)$ be a bi-approximation of $G$. Then the sequence of maps 
\[\lambda_n^\ast\colon \prob(\sub(G))\to \prob(\sub(G))\] converges uniformly,  as $n\to \infty$, to the identity of $\prob(\sub(G))$. In particular every $\mu\in \irs(G)$ is the limit of a  sequence $\nu_n\in \irs(G_n)$, where each $\nu_n$ is $U_n$-saturated.
\end{prop}

\begin{proof}
Arguing as in  the proof of Lemma \ref{l:ts-maps}, it is enough to show that for every sequence $\mu_n\in \prob(\sub(G))$ converging to some limit $\mu\in \prob(\sub(G))$, the sequence $\lambda_n^\ast(\mu_n)$ still converges to $\mu$. To this end, let $f\colon \sub(G)\to \R$ be a continuous function, and observe that
\[|\int fd\lambda_n^\ast(\mu_n)-\int f d\mu|=|\int f\circ \lambda_n d\mu_n-\int f d\mu|\leq |\int fd\mu_n -\int f d\mu|+\int |f\circ \lambda_n-f|d\mu_n.\]
The first term of the last sum tends to 0 because $\mu_n\to \mu$. To bound the second term, observe that $f\circ \lambda_n$ tends to $f$ uniformly, because $f$ is continuous (hence uniformly continuous) and $\lambda_n$ tends to the identity uniformly by Lemma \ref{l:ts-maps}. Therefore $\int |f\circ \lambda_n-f|d\mu_n\le \sup |f\circ\lambda_n-f|\to 0$. Finally the last sentence follows by taking $\nu_n=\lambda^\ast_n(\mu)$.  \qedhere

%Let $\mu\in \irs(G)$. For every $n$ denote by $\mu_n$ the image of $\mu$ through the map $\sub(G)\to \sub(G), H\mapsto [H\cap G_n]^{G_n}_{U_n}$, which  is Borel-measurable by Lemma \ref{l:saturation-map}. The fact that $\mu_n$ is an IRS of $G_n$ which is saturated relatively to ${G_n/U_n}$ is easy to check. 
%To show that $\mu_n\to \mu$, let $f\colon \sub(G)\to \R$ be a continuous function. We have $\int f d\mu_n=\int f d\mu$, where $f_n(H)=f([H\cap G_n]^{G_n}_{U_n})$. Since $f$ is continuous,  Lemma \ref{l:saturation-limits} implies that $f_n(H)\to f(H)$ for all $H\in \sub(G)$. Finally, the functions $f_n$ are uniformly bounded in norm by $\sup |f|$. Therefore, by the dominated convergence theorem,
%\[\int f d\mu_n=\int f_nd\mu\to \int f d\mu,\]
%and it follows that $\mu_n\to \mu$ since $f$ was arbitrary. 
\end{proof}

\begin{thm} \label{thm:no-irs-limit}
If $(G_n, U_n)$ is a bi-approximation of $G$, the following are equivalent:
\begin{enumerate}[label=(\roman*)]

\item \label{i:no-IRS} the group $G$ has no non-trivial IRS;

%\item \label{i:no-IRS-sequence}  for every sequence $\mu_n\in \irs(G_n)$, every cluster point of $(\mu_n)$ in $\prob(\sub(G))$ is a convex combination of $\delta_1$ and $\delta_G$;

\item \label{i:no-IRS-saturated} for every sequence $\mu_n\in \irs(G_n)$ of $U_n$-saturated IRS's, every cluster point of $(\mu_n)$ in $\prob(\sub(G))$ is a convex combination of $\delta_1$ and $\delta_G$. 
\end{enumerate}
\end{thm}

\begin{proof}
The implication \ref{i:no-IRS}$\Rightarrow$\ref{i:no-IRS-saturated} follows from the observation that every cluster point of $\mu_n$ is $G_n$-invariant for arbitrarily high $n$ and therefore is an IRS of $G$. \ref{i:no-IRS-saturated}$\Rightarrow$\ref{i:no-IRS} follows from the last statement of Proposition \ref{p:no-IRS-limit}. \qedhere
\end{proof}

In the particular case when $U_n$ is normal in $G_n$, we obtain:

\begin{cor}\label{c:no-irs-limit-normal}
If $(G_n, U_n)$ is a bi-approximation of $G$ such that each $U_n$ is normal in $G_n$, and the group $G_n/U_n$ has no non-trivial IRS's for all $n\geq 1$, then $G$ has no non-trivial IRS's. 
\end{cor}

\begin{proof}
Under the assumption, every $\mu_n \in \irs(G_n,{G_n/U_n})$ is a convex combination of $\delta_{U_n}$ and $\delta_{G_n}$. Since $U_n\to \{1\}$ and $G_n\to G$ in $\sub(G)$, it follows that every cluster point of such measures is a convex combination of $\delta_1$ and $\delta_{G}$. The statement then follows from Theorem \ref{thm:no-irs-limit}.
\end{proof}

%Next come a criterion to check that a family of IRS's of $G$ is the whole space $\irs(G)$. We will see an application of it in Subsection \ref{s:agw}
%\begin{cor}\label{c:all-irs-limit}
%Let $\Sigma\subset \irs(G)$ be a closed subset such that the map $\lambda_n^\ast|_\Sigma \colon \Sigma \to \irs(G_n,{G_n/U_n})$ is surjective for all $n$ large enough. Then $\Sigma=\irs(G)$.
%\end{cor}
%\begin{proof}
%Assume that there exists $\mu\in \irs(G)\setminus \Sigma$. Let $\mathcal{U}$ be a negbourhood of $\mu$ and $\mathcal{V}$ be a neighbourhood of $\Sigma$ in $\prob(\sub(G))$ such that $\mathcal{U}\cap \mathcal{V}=\varnothing$. By the uniform convergence of $(\lambda^\ast_n)$ to the identity, for $n$ large enough we have $\lambda^\ast_n(\mu)\in \mathcal{U}$ and $\lambda^\ast_n(\nu)\in \mathcal{V}$ for every $\nu\in \Sigma$. But by assumption, for every $n$ there exists $\nu_n\in \Sigma$ such that $\lambda^\ast_n(\mu)=\lambda^\ast_n(\nu_n)$, a contradiciton. \qedhere
%\end{proof}

\subsection{Bi-approximations and  closed subsets of $\sub(G)$}
We now analyse the topological setting.
% Here the notion of being saturated for an element $\H\in \cl(\sub(G))^G$ is slightly less pleasant than in the case of IRS's, because the map $H\mapsto {H}_U^G$ is only upper semicontinuous.
\begin{defin}
Let $G$ be a locally compact group and let $U\le G$ be a compact subgroup. Let $\mathcal{H}\in \cl(\sub(G))^G$. We say that $\mathcal{H}$ is $U$-\emph{saturated} if the set of $U$-saturated subgroups is dense in $\H$. \end{defin}

%$\cl(\sub(G))_{G/U}$

%\begin{defin}
%Let $G$ be a locally compact group and let $U\le G$ be a compact subgroup. Let $\mathcal{H}\in \cl(\sub(G))^G$ be a closed invariant subset. We say that $\mathcal{H}$ is \emph{saturated} relatively to ${G/U}$ if $\{[H]^{G}_{G/U}\colon H\in \H\}$ is a dense subset of $\H$. 
%\end{defin}

\begin{remarks}
\begin{enumerate}[label=(\roman*)]

\item If $\mathcal{H}$ is $U$-saturated, it automatically follows that the set of $U$-saturated subgroups of $\mathcal{H}$ is a dense $G_\delta$ subset of $\H$. This is because the saturation map $H\mapsto {[H]}_{U}^G$, being upper semi-continuous, is continuous on a dense $G_\delta$ subset of $\H$  (see \cite[p.\ 95, Th. 1]{Choquet}, and recall that $\H$ is metrisable here). Since ${[\, \cdot \,]}_{U}^G$ coincides with the identity on a dense subset, it must be the identity on every continuity point.

%since we assume $G$ second countable

\item If $U$ is normal, then $\H\in \cl(\sub(G))^G$ is saturated if and only if it consists of saturated subgroups. This is because in this case the map $H\mapsto [H]^{G}_{G/U}$ is continuous by Lemma \ref{l:saturation-map}. In particular the set of $U$-saturated $\H\in \cl(\sub(G))^G$ is in bijection with $\cl(\sub(G/U))^{G/U}$.
\end{enumerate}
\end{remarks}

%We will need the following lemma. 
%\begin{lem}%Let $\mathcal{H}\in \cl(\sub(G))^G$. Let $\mathcal{K}$ be the closure of $\{[H]^{G}_{G/U}\colon H\in \H\}$. Then $\mathcal{K}\in \cl(\sub(G, G/U))^G$.  Moreover for every $K\in \mathcal{K}$ there exists $H\in \H$ such that $H\le K\le [H]_{U}^G$. 
%\end{lem}
%
%\begin{proof}
%It is clear that $\mathcal{K}$ is closed and invariant. Let $K\in \mathcal{K}$. 
%Let $(H_\nu)\subset \mathcal{H}$ be a net such that $([H_\nu]^{G}_{G/U})$ converges to $K$. Upon taking a subnet we may assume that $(H_\nu)$ converges to some $H\in \mathcal{H}$. From the fact that $H_\nu\le [H_\nu]^{G}_{G/U}$ and from the upper semicontinuity of $H\mapsto [H]^{G}_{G/U}$ we obtain $H\le K\le [H]^{G}_{G/U}$. This proves the last statement. It remains to check that $\mathcal{K}$ is saturated.  Applying $[\cdot ]^{G}_{G/U}$ to $H\le K\le [H]^{G}_{G/U}$ we deduce that $[K]^{G}_{G/U}=[H]^{G}_{G/U}$, and therefore $\{[K]^{G}_{G/U}\colon K\in \mathcal{K}\}$ is a subset of $\mathcal{K}$. The fact that this set is dense is automatic since it contains $\{[H]^{G}_{G/U}\colon H\in \H\}$, which is dense by definition of $\mathcal{K}$.  \qedhere
%\end{proof}
\begin{defin}
Let $(G_n, U_n)$ be a bi-approximation of a locally compact group $G$, with trunc-saturation maps $\lambda_n\colon \sub(G)\to \sub(G)$. We define a sequence of $G_n$-equivariant maps $\bar{\lambda}_n\colon \cl(\sub(G))\to \cl(\sub(G))$ by \[ \bar{\lambda}_n(\mathcal{H}) = \overline{\left\{\lambda_n(H) \, : \, H \in \H \right\}}. \]
\end{defin}

% = \overline{\left\{[H\cap G_n]^{G_n}_{U_n} \, : \, H \in \H \right\}}

%associated to the chain induce a sequence of maps $\bar{\lambda}_n\colon \cl(\sub(G))\to \cl(\sub(G))$, where $\bar{\lambda_n}(\mathcal{H})$ is  the closure of the image of $\H$ under $\lambda_n$. We still call these maps the trunc-saturation maps. 

%Clearly each map $\bar{\lambda}_n$ is $G_n$-equivariant, but not necessarily $G$-equivariant.

We will need the following lemma.

\begin{lem}  \label{l:closure-saturated}
Let $n\geq 1$. Let $\H\in \cl(\sub(G))$. Then for every $K\in \overline{\lambda}_n(\H)$, there exists $H\in \H$ such that $H\cap G_n\le K\le \lambda_n(H)$. 
\end{lem}

%\label{l:saturation-map}

\begin{proof}
Let $(H_k)\subset \mathcal{H}$ such that $(\lambda_n(H_k))$ converges to $K$ as $k \to \infty$ (note that $n$ is fixed).  Upon taking a subsequence we may assume that $(H_k)$ converges to some $H\in \mathcal{H}$. By continuity we have $H_k \cap G_n \to H\cap G_n$ (Lemma \ref{lem-chab-prelim}), and the statement then follows from Lemma \ref{l:saturation-map}. \qedhere
\end{proof}

The following is the analogue of Proposition \ref{p:no-IRS-limit} for closed subspaces of $\sub(G)$ rather than probability measures on $\sub(G)$.

\begin{prop}\label{p:no-URS-limit}
Let $G$ be a locally compact group, and let $(G_n, U_n)$ be a bi-approximation of $G$. Then the sequence of maps \[\bar{\lambda}_n\colon \cl(\sub(G))\to \cl(\sub(G))\] converges uniformly, as $n\to \infty$, to the identity of $\cl(\sub(G))$. In particular, every $\H\in \cl(\sub(G))^G$ is the limit of a sequence $\mathcal{K}_n\in\cl(\sub(G_n))^{G_n}$, where each $\mathcal{K}_n$  is $U_n$-saturated.
\end{prop}

\begin{proof}
Arguing as in the proof of Lemma \ref{l:ts-maps}, we have to show that for every sequence $\H_n\in \cl(\sub(G))$ converging to a limit $\H$, the sequence $\bar{\lambda}_n(\H_n)$ also converges to $\H$ as $n\to \infty$.

Let $H\in \H$, and choose a sequence $H_n\in \H_n$ that converges to $H$. By Lemma \ref{l:ts-maps} the sequence $(\lambda_n(H_n))$ also converges to $\H$. Since $\lambda_n(H_n)\in \overline{\lambda}_n(\H_n)$ and $H\in \H$ was arbitrary, this implies that $\H$ is contained in every cluster point of $\overline{\lambda}_n(\H_n)$.

To show the converse, let $K_n\in \overline{\lambda}_n(\H_n)$ be a sequence converging to some $K \in \sub(G)$. By Lemma \ref{l:closure-saturated}, for every $n$ there exists $H_n\in \H_n$ such that $H_n\cap G_n \le K_n \le \lambda_n(H_n)$. Up to taking a subsequence we may assume that $H_n$ converges to a limit $H$, which  belongs to $\H$ since $\H_n\to \H$. It is easy to see that $H_n\cap G_n$ also tends to $H$. Moreover, applying again Lemma \ref{l:ts-maps} we get that $\lambda_n(H_n)\to H$. Since $H_n\cap G_n \le K_n \le \lambda_n(H_n)$, it follows that $K_n\to H$ as well. Hence $K=H\in \H$. This shows that every cluster point of $(\H_n)$ is contained in $\H$, and therefore $\H_n$ converges to $\H$. \qedhere
\end{proof}

\begin{thm}\label{thm:no-urs-limit}
Given a bi-approximation $(G_n, U_n)$ of a locally compact group $G$, the following are equivalent:
\begin{enumerate}[label=(\roman*)]

\item \label{i:no-URS} the group $G$ has no non-trivial URS's;
%\item \label{i:no-URS-sequence} for every sequence   $\H_n\in \cl(\sub(G_n))^G$, every cluster point of $(\H_n)$ in $\cl(\sub(G))$ contains $\{1\}$ or $G$;
\item \label{i:no-URS-sequence} for every sequence   $\H_n\in \cl(\sub(G_n))^{G_n}$ where each $\mathcal{H}_n$ is $U_n$-saturated, every cluster point of $(\H_n)$ in $\cl(\sub(G))$ contains $\{1\}$ or $G$;

\end{enumerate}
In particular, if $U_n$ is normal in $G_n$ and $G_n/U_n$ has no non-trivial URS's for every $n$, then $G$ has no non-trivial URS's.
\end{thm}

\begin{proof}
Every cluster point of $(\H_n)$ with $\H_n \in \cl(\sub(G_n))^{G_n}$ must be $G$-invariant since $G = \cup G_n$, and hence contains a URS by Zorn's lemma. So \ref{i:no-URS} implies \ref{i:no-URS-sequence}, and the converse follows from Proposition \ref{p:no-URS-limit}.
\end{proof}

\section{Groups of piecewise affine homeomorphisms} \label{s:PL}

The goal of this section is to apply the results of Section \ref{s:as-limits} in order to prove Theorem \ref{thm-intro}.

\begin{defin} \label{def:PL-Qp}
A homeomorphism $g$ of $\Q_p$ is \emph{piecewise affine} if there exists a (possibly infinite) partition of $\Q_p$ into clopen subsets $(\mathcal{U}_i)_{i \geq 0}$ such that for all $i \geq 0$ the restriction  of $g$ on $\mathcal{U}_i$ coincides with the restriction of a map of the form $x \mapsto a_i x + b_i$, with $a_i \in \Q_p^{*}$ and $b_i \in \Q_p$.
\end{defin}

Note that the subsets $\mathcal{U}_i$ in the definition need not be $g$-invariant. We denote by $\mathrm{PL}(\Q_p)$ the group of piecewise affine homeomorphisms of $\Q_p$, and by $\mathrm{PL_c}(\Q_p)$ the (normal) subgroup of $\mathrm{PL}(\Q_p)$ consisting of compactly supported elements.

\subsection{Discrete groups of piecewise affine homeomorphisms}

Before proving Theorem \ref{thm-intro}, we need some intermediate results about certain countable subgroups of $\mathrm{PL_c}(\Q_p)$. The purpose of this paragraph is to establish them, by combining results from \cite{Nek-fp, Dud-Med, LBMB-subdyn}.

%\red{ We shall only provide sketch arguments explaining how to apply the results there in the situation we are interested into.} j'aime pas trop cette phrase...

\begin{defin} \label{def-lambdap}
We denote by $\Lambda_p$ the group of homeomorphisms of $\Z_p$ acting piecewise by maps of the form $x \mapsto ax + b$, with $a \in p^{\Z}$ and $b \in \Z[1/p]$.
\end{defin}

%\begin{defin}
%We denote $\Lambda_p$ the subgroup of $\Gamma_p$ consisting of elements supported inside $\Z_p$.
%\end{defin}

%We implicitly think of $\Z_p$ as the boundary of the coset tree associated to the family of subgroups $(p^n \Z_p)$. If , then , is isomorphic to the Higman-Thompson group $V_p$. Moreover $ \Lambda_p$ is generated by $V_p$ together with the odometer $\alpha: x \mapsto x+1$. 

%belongs to a family of generalisations to of the family of the Higman--Thompson groups

The groups $\Lambda_p$ are member of a larger family of groups studied in \cite{Nek-cuntz, Nek-fp} (with different terminology). To any \emph{self-similar group} $G$ acting on a rooted tree, Nekrashevych associates a group $\mathcal{V}_G$ of homeomorphisms of the boundary of the tree, which contains a copy of the \emph{Higman--Thompson group} $V_p$ whose definition is recalled below \cite[Def.\ 3.2]{Nek-fp}. In the present situation the rooted tree is a $p$-ary tree, whose boundary is identified with $\Z_p$ via $\left\{0,\ldots,p-1\right\}^{\N} \stackrel{\sim}{\rightarrow} \Z_p$, $(a_n) \mapsto \sum a_n p^n$. If $A_p \le \Lambda_p$ is the cyclic subgroup generated by $\alpha_p: x \mapsto x+1$, the action of $A_p$ on $\Z_p$ extends to a self-similar action on the $p$-ary tree, called the \emph{adding machine} in \cite[Ex.\ 3.2]{Nek-fp}. The copy of the Higman--Thompson group $V_p$ inside $\Lambda_p$ consists of all elements that locally preserve the lexicographic order of $\left\{0,\ldots,p-1\right\}^{\N}$, and $\Lambda_p$ is generated by $V_p$ together with $A_p$.

In particular Theorems 4.7 and 4.8 in \cite{Nek-fp} imply the following:

%Let us briefly explain the link. Identify  $\Z_p$ with the boundary of the $p$-ary rooted tree via the obvious identification $\left\{0,\ldots,p-1\right\}^{\N} \stackrel{\sim}{\rightarrow} \Z_p$, $(a_n) \mapsto \sum a_n p^n$.   Let $A\le \Lambda_p$ be the cyclic subgroup generated by  $\alpha_p: x \mapsto x+1$. The action of $A$ on $\Z_p$ extends to self-similar action on the $p$-ary rooted tree which (for $p=2$) is exactly the action in \cite[Example 3.2]{Nek-fp}, where  it is called the \emph{adding machine}.   Note also that $\Lambda_p$ contains a natural copy of the Higman--Thompson's group $V_p$, consisting of all elements $g\in\Lambda_p$ that locally preserve the lexicographic order of $\left\{0,\ldots,p-1\right\}^{\N}$. Under these identifications, the group $\Lambda_p$ coincides with  the group $\mathcal{V}_A$ in the notations of \cite[Definition 3.2]{Nek-fp}. In particular,  Theorems 4.7 and 4.8 in \cite{Nek-fp} imply the following:

%\begin{notation}
%In the sequel we will denote by $\Lambda_p_p =  \Lambda'$ the derived subgroup of $ \Lambda$.
%\end{notation}

%Recall from \cite{C-F-P,Higman} that the abelianization of $V_p$ is trivial for $p=2$ and equal to $\Z / 2\Z$ for odd $p$. For the following result, see Theorem 9.14 in \cite{Nek-cuntz}.
%\red{la simplicit\'e n'est pas prouv\'e la-dedans}

\begin{prop} \label{prop-deriv-simple-nek}
The group $ \Lambda_p'$ is simple, and $ \Lambda_p /  \Lambda_p'$ is isomorphic to $\Z$ if $p=2$ and to $\Z / 2\Z \times \Z$ for odd $p$.

More precisely, the map $x \mapsto (x \mod V_p',0)$ for $x \in V_p$, and $\alpha_p \mapsto (0,1)$, extends to a morphism of $ \Lambda_p$, and induces an isomorphism $ \Lambda_p /  \Lambda_p' \simeq V_p / V_p' \times \Z$.
\end{prop}

%In the proof of the following result, we apply the strategy elaborated in \cite{} to show that 

Dudko and Medynets exhibited sufficient dynamical conditions on a countable group of homeomorphisms ensuring the absence of non-trivial IRS's, and applied their results to the family of Higman-Thompson groups \cite{Dud-Med}. Here we apply their results to the group $ \Lambda_p'$. The arguments are essentially the same as in \cite{Dud-Med}, so we only give a sketch of proof, and refer to \cite[Sec.\ 3]{Dud-Med} for details.

\begin{prop} \label{prop-lambda'-no-irs}
$ \Lambda_p'$ has no non-trivial IRS's.
\end{prop}

\begin{proof}
Let us consider the subgroup $ \Delta \leq  \Lambda_p'$ consisting of elements of $ \Lambda_p'$ acting trivially on a neighbourhood of $0$. One can check that $ \Delta$ can be written as an increasing union of subgroups all isomorphic to $ \Lambda_p'$, so that $ \Delta$ is also a simple group (Proposition \ref{prop-deriv-simple-nek}). Since the action of $ \Delta$ on $\Z_p \setminus \left\{0\right\}$ is easily seen to be compressible \cite[Def.\ 2.5]{Dud-Med}, it follows from \cite[Th. 2.9]{Dud-Med} that $ \Delta$ has no proper character. Now applying \cite[Th. 2.11]{Dud-Med} to the group $ \Lambda_p'$ and the subgroup $ \Delta$ (the verification of the assumption on non-trivial conjugacy classes is easy, and we leave it), we obtain that the group $ \Lambda_p'$ has no proper character either, and hence no non-trivial IRS's.
\end{proof}

The next result is based on an application of \cite[Cor.\ 3.12]{LBMB-subdyn}. We will use the following notation: for a group $G$ acting on a space $X$ and a point $x \in X$, we will write $G_x^0$ for the subgroup of $G$ consisting of elements acting trivially on a neighbourhood of $x$.

Recall that any minimal action of a locally compact group $G$ on a compact space gives rise to a URS of $G$, called the stabilizer URS of the action \cite[Prop.\ 1.2]{Glas-Wei}.

\begin{prop} \label{prop-urs-lambdap}
The stabilizer URS of $ \Lambda_p'$ associated to the action on $\Z_p$ is exactly the collection of subgroups $(\Lambda_p')_x^0$, when $x$ varies in $\Z_p$; and it is the only non-trivial URS of $ \Lambda_p'$.
\end{prop}

\begin{proof}
Given $\gamma \in \Lambda_p'$ and $x \in \Z_p$ a fixed point of $\gamma$, either $x$ is an isolated fixed point, or $\gamma$ acts trivially on a neighbourhood of $x$ (because a non-identity affine map has either $0$ or $1$ fixed point in a given clopen subset of $\Z_p$). Therefore the subgroups in the stabilizer URS are exactly the $(\Lambda_p')_x^0$ according to \cite[Prop. 2.10]{LBMB-subdyn}.

To show that this is the only non-trivial URS, we check that the conditions of \cite[Cor.\ 3.12]{LBMB-subdyn} are satisfied. Consider a coset $x + p^n \Z_p\subset \Z_p$. The subgroup $H$ of $\Lambda_p$  consisting of elements supported in the coset $x + p^n \Z_p$ is isomorphic to $\Lambda_p$.  It follows from the description of the abelianisation map $\Lambda_p\to \Lambda_p/ \Lambda_p'$ given in Theorem 4.8 in \cite{Nek-fp} that its derived subgroup $H'\simeq \Lambda_p'$  coincides with the subgroup of $\Lambda_p'$ consisting of elements supported in the coset $x + p^n \Z_p$. We conclude the latter is isomorphic to $ \Lambda_p'$, and hence is simple by Proposition \ref{prop-deriv-simple-nek}. Moreover point stabilizers for the action on $\Z_p$ are maximal subgroups of $ \Lambda_p'$ (for instance because $ \Lambda_p'$ acts $n$-transitively for all $n$ on each of its orbits in $\Z_p$). Hence we can apply  \cite[Cor.\ 3.12]{LBMB-subdyn} to the action of $ \Lambda_p'$ on $\Z_p$, which gives the conclusion. 
\end{proof}

%We now turn to the construction of groups with no IRS's and no URS's. We begin with the examples considered in the introduction.

\subsection{Locally compact groups of piecewise affine homeomorphisms} \label{subsec:LC-PL}

\begin{defin} \label{def-gamma-PL}
Let $\Gamma_p$ be the subgroup of $\mathrm{PL_c}(\Q_p)$ consisting of elements such that the coefficients $a_i,b_i$ from Definition \ref{def:PL-Qp} satisfy $a_i \in p^{\Z}$ and $b_i \in \Z[1/p]$ for all $i$. In other words, $\Gamma_p$ is the set of $g$ such that there exist disjoint compact open subsets $\mathcal{U}_1,\ldots,\mathcal{U}_n \subset \Q_p$ such that $g$ is supported in the union, and $g$ acts on $\mathcal{U}_i$ by $x \mapsto a_i x + b_i$ with $a_i \in p^{\Z}$ and $b_i \in \Z[1/p]$.
\end{defin}

%Fix a prime $p$. %Whenever $\X$ is a clopen subset of $\Q_p$, we freely identify $\mathrm{Homeo}(\X)$ with the subgroup of $\mathrm{Homeo}(\Q_p)$ acting trivially outside $\X$.
For $n \geq 0$, we denote by $\X_n$ the set of $p$-adic numbers with valuation equal to $-(n+1)$, i.e.\ the complement of $p^{-n} \Z_p$ in $p^{-(n+1)} \Z_p$. 

% $\mathrm{Homeo}(\Q_p)$ such that $F_n$ acts on $\X_n$ piecewise by a map of the form $x \mapsto x + t$, with $t \in \Z[1/p]$ (each piece being a clopen subset of $\X_n$, and with finitely many pieces), and $F_n$ is trivial outside $\X_n$. 

%\[ \X_n = p^{-(n+1)} \Z_p \setminus p^{-n} \Z_p = \bigcup_{\alpha=1}^{p-1} \left( \alpha p^{-(n+1)} + p^{-n} \Z_p \right).\]

%Note that $\X_n$ and $\X_m$ are disjoint whenever $n \neq m$, and $\cup_n \X_n$ is the complement of $\Z_p$ in $\Q_p$. 

%For all $n \geq 0$, 

%(Note at this point that every finite group can be faithfully realized as a group of homeomorphisms of this form.)

\begin{defin} \label{def-G_F}
Given a family $\F = (F_n)$ of finite subgroups $F_n \leq \Gamma_p$ such that $F_n$ is supported inside $\X_n$, we denote by $G_{\F} \leq \mathrm{PL}(\Q_p)$ the homeomorphisms $g$ such that:
\begin{enumerate}[label=(\roman*)]
	\item there exists $N \geq 0$ such that $g$ preserves $p^{-N} \Z_p$, and the restriction of $g$ to $p^{-N} \Z_p$ is piecewise $ax + b$, with $a \in p^{\Z}$ and $b \in \Z[1/p]$;
	\item the restriction of $g$ to $\X_n$ coincides with an element of $F_n$ for all $n \geq N$.
\end{enumerate}
\end{defin}

%\red{Topology}

%The subgroup of $G_{\F}$ of compactly supported homeomorphisms does not depend on the family $\F$; we denote it by $\Gamma_p$.
%Note also that $G_{\F} \cap \mathrm{PL_c}(\Q_p) = \Gamma_p$,

Note that $\Gamma_p$ is a subgroup of $G_{\F}$, and $\Gamma_p$ is actually the group of compactly supported elements of $G_{\F}$. Note that since the $F_n$ have disjoint support, we can naturally view the product $\prod F_n$ as a subgroup of $\mathrm{PL}(\Q_p)$, and it follows from Definition \ref{def-G_F} that $G_{\F}$ is generated by $\Gamma_p$ and $\prod F_n$.

We endow $G_{\F}$ with the topology for which the sets $(\mathcal{V}_{N}(g))_{N \geq 1}$, with $\mathcal{V}_{N}(g) = g \prod_{n\geq N} F_n$, form a basis of neighbourhoods of $g \in G_{\F}$. Since every $g \in G_{\F}$ normalizes all but finitely many $F_n$, this indeed defines a group topology.

%Enlever le paragrape rouge ? \red{ Clearly we have $\Gamma_p = \cup \Gamma_p_n$, where $\Gamma_n$ is supported inside $p^{-n} \Z_p$. The isomorphism class of $\Gamma_n$ is easily seen not to depend on $n$. This group has been studied by Nekrashevych \cite{}, who showed in particular that $\Gamma_{n}$ has infinite abelianization, but that the derived subgroup $\Gamma_{n}'$ is simple. Similar arguments show that} The group $G_{\F}$ has infinite abelianization, hence the necessity to pass to a smaller subgroup in view of Theorem \ref{thm-intro-no-i-u}.

%(ou $\Gamma'$, c'est pareil)

\begin{thm} \label{thm-homeo-Qp}
Assume that $F_n \leq \Gamma_p'$ for all $n \geq 0$. Then the (open, normal) subgroup of $G_{\F}$ generated by $\Gamma_p'$ and $\prod F_n$ has no non-trivial URS's and no non-trivial IRS's.
\end{thm}

%\begin{remark}
%If the groups $F_n$ are uniformly perfect (i.e. every element of $F_n$ can be written as a product of a bounded number of commutators, with a bound which does not depend on $n$), then it can be checked that the subgroup generated by $\Gamma_p'$ and $\prod F_n$ coincides with $G_{\F}'$. However in general $G_{\F}'$ is not closed in $G_{\F}$, and the subgroup generated by $\Gamma_p'$ and $\prod F_n$ is the closure of $G_{\F}'$.
%\end{remark}

\begin{remark}
It is easy to see that any finite alternating group can be realized as a group $F_n$ as in Definition \ref{def-G_F}. Consequently the same holds for any finite group, so that Theorem \ref{thm-intro} follows from Theorem \ref{thm-homeo-Qp}.
\end{remark}

\begin{proof}[Proof of Theorem \ref{thm-homeo-Qp}]
We denote by $G$ the subgroup of $G_{\F}$ generated by $\Gamma_p'$ and $\prod F_n$. For $n\geq 0$, we let $\Gamma_{p,n}$ be the subgroup of $\Gamma_p$ supported in $p^{-n} \Z_p$. Observe that all $\Gamma_{p,n}$ are isomorphic to each other (actually conjugated in $\mathrm{PL}(\Q_p)$); and isomorphic to the group $\Lambda_p$ from Definition \ref{def-lambdap}. Note also that $\Gamma_p$ is the increasing union of all $\Gamma_{p,n}$; and hence $\Gamma_p'$ is also the increasing union of $\Gamma_{p,n}'$.

%increase and ascend to $\Gamma_p'$. At this point we observe that all $\Gamma_{p,n}$ are isomorphic to 

Let $G_n$ be the subgroup of $G$ generated by $\Gamma_{p,n}'$ and $U_n = \prod_{k \geq n} F_k$. Note that since $\Gamma_{p,n}'$ and $U_n$ have disjoint support, the subgroup $G_n$ splits as a direct product $\Gamma'_{p,n} \times U_n$. 

\begin{lem}
$(G_n,U_n)$ is a bi-approximation of the group $G$.
\end{lem}

\begin{proof}
That $G_n$ is open in $G$ is clear. We have to argue that the sequence $(G_n)$ is increasing. The description of $\Gamma_{p,n+1}'$ inside $\Gamma_{p,n+1}$ given in Theorem 4.8 from \cite{Nek-fp} implies the equality $\Gamma_{p,n+1} \cap \Gamma_{p}' = \Gamma_{p,n+1}'$. Since $F_n \leq \Gamma_{p}'$ by our assumption and $F_n \leq \Gamma_{p,n+1}$ by definition, we deduce that $F_n$ lies inside $\Gamma_{p,n+1}'$. Therefore $G_{n+1}$ contains both $F_n$ and $U_{n+1}$, and hence $U_n$. Since $G_{n+1}$ also contains $\Gamma_{p,n}'$, we have verified the inclusion $G_n \leq G_{n+1}$. Note that $\cup G_n$ contains both $\Gamma_p'$ and $\prod F_n$, and hence is equal to $G$. The condition on $U_n$ is also satisfied because $(U_n)$ is decreasing and $\cap U_n = 1$ (see Remark \ref{rmq-case-Un-monot}).
\end{proof}

Since $U_n$ is a normal subgroup of $G_n$ and the discrete quotient $G_n / U_n$ is isomorphic to $\Lambda_p'$, $G_n / U_n$ has no non-trivial IRS's by Proposition \ref{prop-lambda'-no-irs}. We can therefore apply Corollary \ref{c:no-irs-limit-normal} to the bi-approximation $(G_n,U_n)$ of $G$, and deduce that $G$ has no non-trivial IRS's.

%As for the case of IRS's, we will use the bi-approximation $(G_n,U_n)$, and apply Theorem \ref{thm:no-urs-limit}.

%Let $\H\in \cl(\sub(G))^G$ be a cluster point of $(\H_n)$, and up extracting we assume that $\H_n$ tends to $\H$.

We now have to prove the absence of non-trivial URS's. Recall that since $U_n$ is normal in $G_n$, the subgroups of $G_n$ which are $U_n$-saturated are the subgroups containing $U_n$. For $n\geq 1$, let $\H_n \in \cl( \sub(G_n))^{G_n}$ be a closed $G_n$-invariant $U_n$-saturated subset. Let $\K_n \in \cl(\sub(\Gamma_{p,n+1}'))^{\Gamma_{p,n+1}'}$ such that $\H_n$ is the preimage of $\K_n$ under the quotient map. The subset $\K_n$ must contain a URS of $\Gamma_{p,n+1}'$, so we deduce from Proposition \ref{prop-urs-lambdap} that $\K_n$ contains at least one of $\left\{1\right\}$, or $\Gamma_{p, n+1}'$, or the stabilizer URS associated to the action of $\Gamma_{p, n+1}'$ on $p^{-n} \Z_p$. In terms of $\H_n$, this implies that $\H_n$ contains $U_n$, or $G_n$, or $(G_n)_{x_n}^0$ with $x_n = p^{-n}$ (recall that the notation $G_x^0$ has been defined before Proposition \ref{prop-urs-lambdap}). 

\begin{lem} \label{lem-fat-stab-conv-G}
Let $x_n \in \Q_p$ such that $v_p(x_n) \rightarrow - \infty$. Then $(G_n)_{x_n}^0 \rightarrow G$ in $\sub(G)$.
\end{lem}
%We have to check that every $g \in G$ is a limit of $g_n$ with $g_n \in (G_n)_{x_n}^0$.

%there is $n_0$ such that $g \in G_n$ for all $n \geq n_0$; and

\begin{proof}
Write $k_n = v_p(x_n)$. Given $g \in G$, we define an element $g_n$ by declaring that $g_n$ acts on $p^{k_n+1} \Z_p$ like $g$, and trivially elsewhere. Since $g$ preserves $p^{k_n+1} \Z_p$ eventually, the element $g_n$ is well defined, $g_n \in (G_n)_{x_n}^0$ and $g g_n^{-1} \in \prod_{k\geq-k_n-1} F_k$. Therefore $(g_n)$ converges to $g$, and the lemma is proved.
\end{proof}

Now assume that $(\H_{n_k})$ is a sequence of closed $G_{n_k}$-invariant subsets saturated relatively to $G_{n_k} / U_{n_k}$, such that $(\H_{n_k})$ converges to $\H$ in $\F(\sub(G))$. If $I_1,I_2,I_3$ are respectively the set of integers $k$ such that $U_{n_k} \in \H_{n_k}$, $G_{n_k} \in \H_{n_k}$ and $(G_{n_k})_{x_{n_k}}^0 \in \H_{n_k}$; then according to the previous paragraph at least one of $I_1,I_2,I_3$ is infinite. Since $U_n \rightarrow \left\{1\right\}$, $G_n \rightarrow G$ and $(G_n)_{x_n}^0 \rightarrow G$ by Lemma \ref {lem-fat-stab-conv-G}, we deduce that $\H$ must contain at least one of $\left\{1\right\}$ or $G$. Therefore we are in position to apply Theorem \ref{thm:no-urs-limit}, which shows that $G$ has no non-trivial URS's.
\end{proof}

\section{More examples} \label{s:more}
 
\subsection{Locally elliptic groups} \label{subsec:permutations}

Our goal in this paragraph is to provide other (and maybe more tractable) examples of non-discrete groups with no non-trivial URS's. These are based on a variation of a construction independently due to Willis \cite[Sec.\ 3]{Willis-co} and Akin--Glasner--Weiss \cite[Sec.\ 4]{A-G-W}, and also considered by Caprace--Cornulier in \cite{Cap-Cor}. As in \cite{Willis-co,A-G-W}, the groups will be defined as groups of permutations which are prescribed on certain blocks; the difference here being that we force the fixed point set of every permutation to be infinite.

The construction goes as follows. Let $k_n$ be a strictly increasing sequence of natural numbers, with $k_0=0$. For all $n\geq 0$, choose a finite group $D_n\le \alt(\{k_n,\ldots, k_{n+1}-1\})$, and see the compact group $\prod_{n\geq 0} D_n$ as a permutation group on $\Z$ acting on $\Z_{\geq 0}$ on each interval $[k_n, k_{n+1}-1]$ via $D_n$, and acting trivially on $\Z_{< 0}$. Consider the group $G$ of permutations of $\Z$ generated by $\alt_f(\Z)$ (the group of finitary even permutations of $\Z$) and $\prod D_n$. Equivalently, $G$ can be described as the group of permutations moving only finitely many negative integers, and acting on all but finitely many intervals $[k_n, k_{n+1}-1]$ like an element of $D_n$. Since every finitary permutation centralizes all but finitely many $D_n$, the group $G$ carries a locally compact group topology for which the inclusion $\prod D_n \hookrightarrow G$ is a homeomorphism onto its image. 

%he group $\prod_{n\geq 0} D_n$ is commensurated by $G$, therefore

\begin{thm} \label{thm-permu-LN-urs}
For every sequence $(D_n)$, the following hold:
\begin{enumerate}[label=(\roman*)]
\item \label{item-permu-no-urs} $G$ has no non-trivial URS's;
\item \label{item-permu-no-cofinite} $G$ has no proper subgroup of finite covolume.
\end{enumerate}
\end{thm}

Recall that a if $H$ is a closed subgroup of a locally compact group $G$, we say that $H$ has finite covolume if $G/H$ carries a $G$-invariant probability measure.

Before giving the proof of Theorem \ref{thm-permu-LN-urs}, let us make several observations:

\begin{remark}
\begin{enumerate}[label=(\roman*)]
\item In contrast with the groups considered in Section \ref{s:PL}, the group $G$ does admit non-trivial IRS's, that arise from random invariant partitions of $\Z$ (as in the case of the alternating group \cite{Ver:IRS}).
\item A common feature between these groups and the ones from Section \ref{s:PL} is that every single element normalizes a compact open subgroup. Equivalently, Willis' scale function is identically equal to one \cite{Wil94}. However an important difference is that here $G$ is locally elliptic (i.e.\ every finite subset generates a relatively compact subgroup), while the groups from Definition \ref{def-G_F} have plenty of infinite finitely generated discrete subgroups. 
\end{enumerate} 
\end{remark}

As in the previous section, the proof of the absence of URS's in $G$ will appeal to auxiliary results about discrete groups, namely:

\begin{prop} \label{prop:no-urs-alternating}
For an infinite set $\Omega$, the group $\alt_f(\Omega)$ does not admit non-trivial URS's.
\end{prop}

Here $\alt_f(\Omega)$ is the group of finitely supported alternating permutations of $\Omega$. Proposition \ref{prop:no-urs-alternating} was proved by Thomas and Tucker-Drob {in \cite{Th-TD}} (for $\Omega$ countable), where it is deduced from Vershik's classification of the IRS's of this group \cite{Ver:IRS}. A closely related and more precise result goes back to Sehgal--Zalesskii \cite{Se-Za-alt}, who characterized the subgroups of $\alt_f(\Omega)$ whose conjugacy class avoids an open neighbourhood of $\{1\}$ in the Chabauty space $\sub(\alt_f(\Omega))$. See the equivalence (ii) $\Leftrightarrow$ (iii) from Theorem 1 in \cite{Se-Za-alt} (from which Proposition \ref{prop:no-urs-alternating} easily follows).

\begin{proof}[Proof of Theorem \ref{thm-permu-LN-urs}]
For $n \geq 1$, consider the open subgroup \[ G_n=\langle \alt_f(\Z_{< k_n})\cup \prod D_j\rangle, \] which consists of elements of $G$ acting on $[k_i, k_{i+1}-1]$ like an element of $D_i$ for all $i \geq n$. The sequence $(G_n)$ is increasing and ascends to $G$, and if we write $U_n=\prod_{j\geq n} D_j \le G_n$, then $(U_n)$ is decreasing and $\cap U_n = 1$. So $(G_n, U_n)$ is a bi-approximation of $G$. Note that $U_n$ is normal in $G_n$ and that $G_n \cong \alt_f(\Z_{< k_n}) \times U_n$, so that the quotient $G_n/U_n \cong \alt_f(\Z_{< k_n})$ has no non-trivial URS's by Proposition \ref{prop:no-urs-alternating}. Therefore statement \ref{item-permu-no-urs} follows by applying Theorem \ref{thm:no-urs-limit}.

For \ref{item-permu-no-cofinite}, let $H \leq G$ be a closed subgroup of finite covolume. Since $G_n$ is open in $G$ for all $n \geq 1$, $H \cap G_n$ has finite covolume in $G_n$. The subgroup $U_n$ being compact, we deduce that the projection of $H \cap G_n$ to $G_n / U_n$ is a closed subgroup of finite covolume. But $G_n / U_n \cong \alt_f(\Z_{< k_n})$ is a discrete infinite simple group, and hence has no proper finite index subgroups. Therefore $H \cap G_n$ surjects onto $G_n / U_n$, and since $(G_n)$ ascends to $G$ and $(U_n)$ decreases, we actually have $G = H U_n$ for arbitrary $n$. This shows that $H$ is dense in $G$. Since $H$ is also closed by assumption, we must have $H=G$.
\end{proof}
 
\subsection{Groups of infinite matrices}

A result of Kirillov \cite{Kir-char} (extended by Peterson and Thom in \cite{Pet-Thom}) says that the group $\psl(d, \Q)$ has no non-trivial characters for $d\geq 3$, and therefore no non-trivial IRS's \cite[Th.\ 2.11]{Dud-Med}, \cite[Th.\ 3.2]{Pet-Thom}:

\begin{prop}\label{prop:pet-thom}
For $d\geq 3$, the group $\psl(d, \Q)$ has no non-trivial IRS's.
\end{prop}

In this paragraph we use this fact to describe additional examples of non-discrete groups without IRS's. The construction is in spirit very close to the one carried out in \cite[Prop.\ 3.5]{Willis-co}.

Let $\Q^{(\N)}$ be the infinite dimensional vector space over $\Q$ with basis $(e_i)_{i\in \N}$. Let $\sl(\infty, \Q)=\lim_d \sl(d, \Q)$ be the group of linear isomorphisms of $\Q^{(\N)}$ that fix all but finitely many $e_i$, and with determinant one. Let $(d_n)$ be a strictly increasing sequence of natural numbers, and for convenience we will assume that $d_n$ is odd for infinitely many $n$. Write $k_n = d_{n+1}-d_n$. We choose a finite subgroup $D_n\le \sl(k_n, \mathbb{Q})$ for all $n$, and view the group $\prod_{n\geq 0} D_n$ as a group of block-diagonal linear transformations of $\Q^{(\N)}$, where each $D_n$ acts on the subspace spanned by $\{e_{d_n+1}, \ldots, e_{d_{n+1}}\}$ and fixes all other elements of the basis. Finally we consider the group $G=\langle \sl(\infty, \Q)\cup \prod_{n\geq 0} D_n\rangle$, equipped with the topology for which  the inclusion of $\prod D_n$ is continuous and open.

\begin{prop} \label{prop-psl-non-discret}
For every choice of $(D_n)$, the group $G$ has no non-trivial IRS's.
\end{prop}

We note however that we do not know whether these groups admit non-trivial URS's. The proof of Proposition \ref{prop-psl-non-discret} will again follow the same scheme, which consists in applying Corollary \ref{c:no-irs-limit-normal} to a suitable bi-approximation of $G$.

\begin{proof}[Proof of Proposition \ref{prop-psl-non-discret}]
If we denote by $G_n$ the subgroup of $G$ generated by $\sl(d_n, \Q)$ and $\prod_{j\geq 0} D_j$, and by $U_n=\prod_{j\geq n} D_j \leq G_n$, we easily verify that $(G_n, U_n)$ forms a bi-approximation of the group $G$. Moreover $U_n$ is normal in $G_n$ and $G_n/U_n \cong \sl(d_n , \Q)$.  Since $d_n$ is odd for infinitely many $n$, up to taking a subsequence, $G_n/U_n$ has no nontrivial IRS's by Proposition \ref{prop:pet-thom}. Therefore the conclusion of Proposition \ref{prop-psl-non-discret} follows by applying Corollary \ref{c:no-irs-limit-normal}.
\end{proof}

\bibliographystyle{amsalpha}
\bibliography{bib-no-urs}

\end{document}